\documentclass[12pt,a4paper]{amsart}
\usepackage{tikz,xcolor}
\usepackage{multirow}
\newtheorem{theo+}              {Theorem}           [section]
\newtheorem{prop+}  [theo+]     {Proposition}
\newtheorem{coro+}  [theo+]     {Corollary}
\newtheorem{lemm+}  [theo+]     {Lemma}
\newtheorem{exam+}  [theo+]     {Example}
\newtheorem{rema+}  [theo+]     {Remark}
\newtheorem{defi+}  [theo+]     {Definition}
\newenvironment{theorem}{\begin{theo+}}{\end{theo+}}

\newenvironment{corollary}{\begin{coro+}}{\end{coro+}}

\usepackage{amsthm}
\theoremstyle{plain} \theoremstyle{remark}
\newtheorem{remark}{Remark}

\def \r{\mbox{${\mathbb R}$}}

\def\E{/\kern-1.0em \equiv }

\author{}

\begin{document}
\title[Stability and the index of  biharmonic hypersurfaces in a Riemannian manifold]{Stability and the index of  biharmonic hypersurfaces in a Riemannian manifold}
\author{Ye-Lin Ou $^{*}$}
\thanks{$^{*}$ This work was supported by a grant from the Simons Foundation ($\#427231$, Ye-Lin Ou).}
\address{Department of
Mathematics,\newline\indent Texas A $\&$ M University-Commerce,
\newline\indent Commerce, TX 75429,\newline\indent USA.\newline\indent
E-mail:yelin.ou@tamuc.edu }
\date{02/08/2020}
\maketitle
\section*{Abstract}
\begin{quote}  In this paper, we give an explicit second variation formula for  a biharmonic hypersurface in a Riamannian manifold similar to that of a minimal hypersurface. We then use the second variation formula to compute the stability index of the known biharmonic hypersurfaces in a Euclidean sphere, and to prove the non-existence of unstable proper biharmonic hypersurface in a Euclidean space or a hyperbolic space, which adds another special case to support Chen's conjecture on biharmonic submanifolds.
{\footnotesize } 
\end{quote}

\section{Stability and the  index of minimal hypersurfaces}

It is well known that minimal hypersurfaces $M^m\to (N^{m+1}, h)$ in a Riemannian manifold  are critical points of the area functional on hypersurfaces, i.e.,
\begin{equation}\notag
\frac{\rm d}{{\rm d}t}\big({\rm Area} (M_t)\big)_{t=0}=-m\int_{M} f H{\rm d}v_g=0.
\end{equation}
This is equivalent to the statement that the mean curvature $H=\frac{1}{m} {\rm Tr} A$ of the hypersurface of $M$ vanishes identically, where $A$ is the shape operator of  the hypersurface.

 As it is also well known that a critical point may not give a local minimum of the area functional. To have a better understanding of minimal hypersurfaces as the critical points of a functional, one needs to know the second variation that leads to the concepts of the stability and the index of minimal hypersurfaces. \\
 
 Recall (see e.g., \cite{Al} and \cite{MR17}) that a minimal hypersurface is {\em stable} if the second variation of the area functional is always nonnegative for  any normal variation with compact support, i.e., 
 \begin{equation}\notag
\frac{\rm d^2}{{\rm d}t^2}\big({\rm Area} (M_t)\big)_{t=0}\ge 0.
\end{equation}
 
 For a complete orientable minimal hypersurface $M^m\to (N^{m+1},h)$ in a Riemannian manifold,  there is a unit normal vector field $\xi$ along $M$  so that any section $V$ of the normal bundle with compact support can be written as $V=f\xi$ for a function $f$ with compact support in $M$, and the second variation of the area functional with the $V=f\xi$ as variation vector field can be written as :
\begin{equation}\label{Ms}
\frac{\rm d^2}{{\rm d}t^2}\big({\rm Area} (M_t)\big)_{t=0}=\int_{M}\{|\nabla f|^2-({\rm Ric}^N(\xi,\xi)+|A|^2)f^2\}{\rm d}v_g,
\end{equation}
where $|A|^2$ is the squared norm of the second fundamental form of the hypersurafce, and ${\rm Ric}^N(\xi,\xi)=\sum_{i=1}^m\langle{\rm R}^N(\xi, e_i)e_i, \xi\rangle=\sum_{i=1}^m {\rm R}^N (\xi, e_i, \xi, e_i)$ is the Ricci curvature in the direction $\xi$.\\

Note that by using the divergence theorem: $\int_Mf\Delta f {\rm d}v_g= -\int_M|\nabla f|^2 {\rm d}v_g$, we can rewrite  (\ref{Ms})  as
\begin{equation}\label{Ms1}
\frac{\rm d^2}{{\rm d}t^2}\big({\rm Area} (M_t)\big)_{t=0}=q_{M}(f)=-\int_{M} f J (f ){\rm d}v_g\ge 0,
\end{equation}
where  $J (f )=\Delta f+(|A|^2+ {\rm Ric}^N(\xi,\xi))f$ is called the {\it Jacobi operator} on the minimal hypersurface.\\

Recall (see e.g., \cite{Al}) that the {\it index} of a minimal hypersurface $M$, denoted by ${\rm Ind}(M)$, is  the maximum dimension of any subspace $V$ of $C^{\infty}_0(M)$ on which $q_M(f)$ is negative, i.e.,
$${\rm Ind}(M)= {\rm Max} \{ {\rm dim} V: V\subset C^{\infty}_0(M)\,| \,q_M(f)<0, \forall\, f\in V\}.$$

In particular, the index of a minimal hypersurface $M\hookrightarrow S^{m+1}$ in a Euclidean sphere is  the  the largest dimension of subspace $V\subset C_0^{\infty}(M)$ on which the quadratic form 
\begin{align}
q_{M}(f)=-\int_{M} f [\Delta f+(|A|^2+m)f]{\rm d}v_g< 0.
\end{align}

The following are some well known facts about the index of minimal hypersurfaces in a sphere:
\begin{itemize}
\item For a compact minimal hypersurface $M$ in $S^{m+1}$,  ${\rm Ind}(M)\ge 1$ and with $``="$ holds if and only if $M$ is a totally geodesic equator $S^m\subset S^{m+1}$(\cite{Si68});
\item For a compact non-totally geodesic minimal hypersurface $M^m\to S^{m+1}$, ${\rm Ind}(M)\ge m+3$ (see \cite{Ur90} for $m=2$ and \cite{So93} for the general case);
\item For the minimal Clifford torus $S^p(\sqrt{\frac{p}{m}})\times S^q(\sqrt{\frac{q}{m}})\hookrightarrow S^{m+1}$ with $p+q=m$, the index ${\rm Ind}(M)= m+3$;
\item It has been a  conjecture which is still open (see e.g., \cite{Al}, \cite{ABP07})  that any compact non-totally geodesic minimal hypersurface $M^m\to S^{m+1}$ with ${\rm Ind}(M)= m+3$ is a Clifford torus.
\end{itemize}

Biharmonic hypersurfaces are generalizations of minimal hypersurfaces. A biharmonic hypersurface in a Riemannian manifold can be characterized as an isometric immersion $M^m\to (N^{m+1}, h)$ whose mean curvature function $H$ solves the following equation (see  \cite{Ji87}, \cite{Ch91}  and \cite{CMO02} for the case when the ambient space is a space form, and \cite{Ou10} for the general case):
\begin{equation}\label{BHEq}
\begin{cases}
\Delta H-H |A|^{2}+H{\rm
Ric}^N(\xi,\xi)=0,\\
 2A\,({\rm grad}\,H) +\frac{m}{2} {\rm grad}\, H^2
-2\, H \,({\rm Ric}^N\,(\xi))^{\top}=0,
\end{cases}
\end{equation}
where ${\rm Ric}^N : T_qN\longrightarrow T_qN$ denotes the Ricci
operator of the ambient space defined by $\langle {\rm Ric}^N\, (Z),
W\rangle={\rm Ric}^N (Z, W)$ and  $A$ is the shape operator of the
hypersurface with respect to the unit normal vector $\xi$.\\

It is clear from (\ref{BHEq}) that any minimal hypersurface is automatically a biharmonic hypersurface. So it is a custom to call a biharmonic hypersurface which is not minimal  a {\it proper biharmonic} hypersurface. For more  study of biharmonic maps and biharmonic submanifolds we refer the reader to a recent book \cite{OC20} and the references therein.\\

In this paper, we derive an explicit second variation formula for  a biharmonic hypersurface in a Riemannian manifold similar to that of a minimal hypersurface. We then use the second variation formula to compute the stability index of the known biharmonic hypersurfaces in a Euclidean sphere, and to prove the non-existence of unstable proper biharmonic hypersurface in a Euclidean space or a hyperbolic space, which adds another special case to support Chan's conjecture on biharmonic submanifolds.

\section{Stability and the index of biharmonic hypersurfaces}
In light of the ideas from the study of stability and the index of minimal hypersurfaces, we define a proper biharmonic hypersurface $M\to (N^{m+1},h)$ to be {\bf stable} if the second variation of the bienergy functional is always nonnegative for  any normal variation with compact support.  With this, we have
\begin{theorem}\label{MT1} 
A complete orientable biharmonic hypersurface $\phi:M^m\to (N^{m+1},h)$ of a Riemannian manifold is stable if and only if for any compactly supported function $f$ on $M$, we have
\begin{align}\notag
Q(f)=&\frac{{\rm d}^2 }{{\rm d} t^2}E_2(\phi_t)|_{t=0}\\\notag
=&\int_M [f (|A|^2-{\rm Ric}^N(\xi, \xi))-\Delta f]^2  ] {\rm d} v_g\\\label{2Vformula}
&+\int_M |mf\nabla H-2({\rm Ric}^N(\xi))^{\top}+2A(\nabla f)|^2  {\rm d} v_g\\\notag
&+\int_M mf^2H[(\nabla^{N}_{\xi}{\rm Ric}^N)(\xi,\xi))-2{\rm Tr}\,{\rm R}^N(\xi, \cdot, \xi, \nabla^{N}_{\xi}(\cdot))]{\rm d} v_g\\\notag
&-\int_M 4mf^2H {\rm Tr}\,{\rm R}^N(\xi, A(\cdot), \xi, \cdot) {\rm d} v_g\ge 0.
\end{align}
\end{theorem}
\begin{proof}
The  following second variation formula for a general biharmonic map $\phi: (M^m, g)\to (N^n, h)$ between two Riemannian manifolds was derived by Jiang in \cite{Ji86}:
\begin{equation}\label{2Vformula1}
\begin{aligned}
\frac{{\rm d}^2 }{{\rm d} t^2}E_2(\phi_t)|_{t=0}=&\int_M [|J^{\phi}(V)|^2+ {\rm R}^N(V, \tau(\phi), V,\tau(\phi) ] {\rm d} v_g\\
&\hskip-.5in -\sum_{i=1}^m\int_M \langle V, (\nabla^{N}_{{\rm d}\phi(e_i)}{\rm R}^N)({\rm d}\phi(e_i), \tau(\phi))V \\
&\hskip-.5in + (\nabla^{N}_{\tau(\phi)}{\rm R}^N)({\rm d}\phi(e_i), V){\rm d}\phi(e_i)\\
&\hskip-.5in +2{\rm R}^N({\rm d}\phi(e_i), V)\nabla^{\phi}_{e_i}\tau(\phi)+2{\rm R}^N({\rm d}\phi(e_i), \tau(\phi))\nabla^{\phi}_{e_i}V \rangle {\rm d} v_g.
\end{aligned}
\end{equation}
\indent When $\phi:M^m\to (N^{m+1},h)$ is an orientable biharmonic hypersurface, we consider the normal variation with variation vector field $V=f\xi$ and use $\tau(\phi)=mH\xi$, identify ${\rm d}\phi(e_i)=e_i$, and a straightforward computation to have
\begin{align}\label{C1}
&{\rm R}^N(V, \tau(\phi), V,\tau(\phi)=m^2f^2H^2{\rm R}^N(\xi, \xi, \xi,\xi)=0,\\\label{C2}
&\sum_{i=1}^m \langle V,  (\nabla^{N}_{\tau(\phi)}{\rm R}^N)({\rm d}\phi(e_i), V){\rm d}\phi(e_i)\rangle=mfH\sum_{i=1}^m \langle \xi,  (\nabla^{N}_{\xi}{\rm R}^N)(e_i, f\xi)e_i\rangle\\\notag
&=mf^2H[-(\nabla^{N}_{\xi}{\rm Ric}^N)(\xi,\xi))+2{\rm Tr}\,{\rm R}^N(\xi, \cdot, \xi, \nabla^{N}_{\xi}(\cdot))],
\\\label{C3}
&\sum_{i=1}^m \langle V, (\nabla^{N}_{{\rm d}\phi(e_i)}{\rm R}^N)({\rm d}\phi(e_i), \tau(\phi))V\rangle=mf^2H\sum_{i=1}^m \langle \xi, (\nabla^{N}_{e_i}{\rm R}^N)(e_i, \xi)\xi\rangle=0, 
\end{align}
 \begin{align}\ \label{C4}
&\sum_{i=1}^m \langle V, 2{\rm R}^N({\rm d}\phi(e_i), V)\nabla^{\phi}_{e_i}\tau(\phi)\rangle=2mf^2H\sum_{i=1}^m{\rm R}^N(\xi, \nabla^N_{e_i}\xi, e_i,\xi)\\\notag
 &=2mf^2H {\rm Tr}\,{\rm R}^N(\xi, A(\cdot), \xi, \cdot), {\rm and}
\\\label{C5}
&\sum_{i=1}^m \langle V, 2{\rm R}^N({\rm d}\phi(e_i), \tau(\phi))\nabla^{\phi}_{e_i}V \rangle=2mfH\sum_{i=1}^m \langle \xi, {\rm R}^N(e_i, \xi)\nabla^{N}_{e_i}(f\xi) \rangle\\\notag
 &=2mf^2H  {\rm Tr}\,{\rm R}^N(\xi, A(\cdot), \xi, \cdot).
\end{align}
On the other hand, using the formula (see e.g., \cite{Ou09}) 
\begin{align}\label{g9}
J^{\phi}(V)=J^{\phi}(f\xi)=fJ^{\phi}(\xi)-(\Delta f)\xi-2\nabla^{\phi}_{\nabla f}\xi,
\end{align}
and a further computation (see also \cite{Ou10}), we obtain
\begin{align}\notag
J^{\phi}(\xi)=&-\sum_{i=1}^m\left( (\nabla^{\phi}_{e_i}\nabla^{\phi}_{e_i}-\nabla^{\phi}_{\nabla^M_{e_i}e_i})\xi-{\rm R}^N({\rm d}\phi(e_i), \xi){\rm d}\phi(e_i)\right)\\\label{10}
=&(|A|^2-{\rm Ric}^N(\xi, \xi))\xi+m\nabla H-2({\rm Ric}^N(\xi))^{\top},\\\label{11}
2\nabla^{\phi}_{\nabla f}\xi&=-2A(\nabla f).
\end{align}

It follows from (\ref{10}), (\ref{11}) and (\ref{g9}) that
\begin{align}\notag
|J^{\phi}(V)|^2=&|J^{\phi}(f\xi)|^2\\\label{12}
=&[f (|A|^2-{\rm Ric}^N(\xi, \xi))-\Delta f]^2 +|mf\nabla H-2({\rm Ric}^N(\xi))^{\top}+2A(\nabla f)|^2.
\end{align}
Substituting (\ref{C1})-(\ref{C5}) and (\ref{12}) into (\ref{2Vformula1}) we obtain 
\begin{align}\notag
Q(f)=&\frac{{\rm d}^2 }{{\rm d} t^2}E_2(\phi_t)|_{t=0}\\\notag
=&\int_M [f (|A|^2-{\rm Ric}^N(\xi, \xi))-\Delta f]^2  ] {\rm d} v_g\\\label{2Vformula2}
&+\int_M |mf\nabla H-2({\rm Ric}^N(\xi))^{\top}+2A(\nabla f)|^2  {\rm d} v_g\\\notag
&+\int_M mf^2H[(\nabla^{N}_{\xi}{\rm Ric}^N)(\xi,\xi))-2{\rm Tr}\,{\rm R}^N(\xi, \cdot, \xi, \nabla^{N}_{\xi}(\cdot))]{\rm d} v_g\\\notag
&-\int_M 4mf^2H {\rm Tr}\,{\rm R}^N(\xi, A(\cdot), \xi, \cdot) {\rm d} v_g,
\end{align}
from which and the definition of the stability of a biharmonic hypersurface we obtain the theorem.
\end{proof}
\begin{corollary} 
An orientable biharmonic hypersurface $\phi: M^m\to (N^{m+1}(c), h)$ in a space form of constant sectional curvature  $c$ is  stable if and only if the stability inequality
$Q(f) \ge 0$ holds for any compactly supported smooth function $f$ on $M$, where
\begin{equation}\notag
\begin{aligned}
Q(f)=  \int_M \big\{[f (|A|^2-cm)-\Delta f]^2 +|mf\nabla H+2A(\nabla f)|^2-4cm^2f^2H^2\big\}{\rm d}v_g.
\end{aligned}
\end{equation}
 In particular,
(i) any biharmonic hypersurface in a Euclidean space or a hyperbolic space is stable, and
(ii) the stability quadratic form for a biharmonic hypersurface $M^m\to S^{m+1}$ in a Euclidean sphere is given by
\begin{equation}\label{Sm}
\begin{aligned}
Q(f)= \int_M \big\{[f (|A|^2-m)-\Delta f]^2 +|mf\nabla H+2A(\nabla f)|^2-4m^2f^2H^2\big\}{\rm d}v_g.
\end{aligned}
\end{equation}
\end{corollary}
\begin{proof}
The corollary follows from Equation (\ref{2Vformula2}) and the following identities for a space form $N^{m+1} (c)$ of constant sectional curvature $c$: 
\begin{align}\notag
&{\rm Ric}^N(\xi, \xi))=mc,\\\notag
&{\rm Ric}^N(\xi))^{\top}=0,\\\notag
&(\nabla^{N}_{\xi}{\rm Ric}^N)(\xi,\xi)=0,\\\notag
&{\rm Tr}\,{\rm R}^N(\xi, \cdot, \xi, \nabla^{N}_{\xi}(\cdot))=0,\;{\rm and}\\\notag
& 4mf^2H {\rm Tr}\,{\rm R}^N(\xi, A(\cdot), \xi, \cdot) =4cm^2f^2H^2.
\end{align}
\end{proof}

\begin{remark}
For stable minimal surfaces in Euclidean space $\r^3$, we have a well-known result of do-Carmo-Peng: Any complete oriented and stable minimal surface $\phi: M^2\to \r^3$ is a plane. On the other hand,  we know that catenoid is a complete  and oriented minimal surfaces in $\r^3$, so it is unstable. In contrast, our corollary above says that there is no unstable biharmonic hypersurface in a Euclidean space or a hyperbolic space. This adds another special case to support Chen's conjecture on biharmonic submanifolds which can be stated as there exists no proper biharmonic submanifold in a Euclidean space. See \cite{OC20} for a more detailed account on Chen's conjecture on biharmonic submanifolds.
\end{remark}

\section{The stability index of biharmonic hypersurfaces in $S^{m+1}$}

Again, following the idea  of the index of minimal hypersurfaces, we define the {\bf index} of a proper biharmonic hypersurface $M\to (N^{m+1},h)$ to be  the maximum dimension of any subspace $V$ of $C^{\infty}_0(M)$ on which $Q(f)$ defined in (\ref{2Vformula2}) is negative, i.e.,
$${\rm Ind}(M)= {\rm Max} \{ {\rm dim} V: V\subset C^{\infty}_0(M)\,| \,Q(f)<0, \forall\, f\in V\}.$$

About  biharmonic hypersurfaces in a sphere, we know that
\begin{itemize}
\item a hypersurface $\varphi : (M^m, g) \longrightarrow  S^{m+1}$ 
with  nonzero constant mean curvature is biharmonic if and only if the squared norm of the shape operator is constant (see  \cite{Ji86} or directly using (\ref{BHEq})).
\item the only known proper biharmonic hypersurfaces in a sphere are (\cite{Ji86}, \cite{CMO01}): $S^m(\frac{1}{\sqrt{2}})$ and $S^p(\frac{1}{\sqrt{2}})\times S^{m-p}(\frac{1}{\sqrt{2}})$ for $p\ne \frac{m}{2}$, or an open part of one of these two. It has been a conjecture (\cite{BMO1}) which is still open that  there is no other proper biharmonic hypersurface in a sphere than open parts of these two.
\end{itemize}

In this section, we will use the stability form (\ref{Sm}) to compute the index of  the known proper biharmonic hypersurfaces in a sphere.

\begin{theorem}\label{MT2}
(i) For $1\le p <q=m-p$, the stability index of the proper biharmonbic hypersurface  $S^p(\frac{1}{\sqrt{2}})\times S^q(\frac{1}{\sqrt{2}}) \to S^{m+1}$ is 
\begin{equation}\label{indF}
 {\rm Ind}\left(S^p(\frac{1}{\sqrt{2}})\times S^{q}(\frac{1}{\sqrt{2}})\right)=
 \begin{cases}
 1,\;\;\;\;\; {\rm for}\; 1\le p< q\le 2p,\\
 p+2, \;\;\;\;\; {\rm for}\; 2p< q.
 \end{cases} 
\end{equation}
(ii)  The stability index of the biharmonic hypersurface $S^m(\frac{1}{\sqrt{2}})\to S^{m+1}$ is 
\begin{align}
{\rm Ind}\,\left(S^m(\frac{1}{\sqrt{2}})\right)=1.
\end{align}
\end{theorem}
\begin{proof}
It is also known (cf. e.g., \cite{Al}, \cite{ABP07})  that if $\lambda$ is an eigenvalue of the Laplacian on $S^p(r_1)$ with multiplicity $m_{\lambda}$ and $\mu$ is an eigenvalue of the Laplacian on $S^{m-p}(r_2)$ with multiplicity $m_{\mu}$, then $\nu=\lambda+\mu$ is an eigenvalue of the Laplacian on the product $S^p(r_1)\times S^{m-p}(r_2)$ with with a multiplicity $\sum m_{\lambda} m_{\mu}$ where the sum is made over all possible $\lambda, \mu$ satisfying $\lambda+\mu=\nu$.

In particular, when $r=1/\sqrt{2}$, the proper biharmonic hypersurface $T^{p,q}:=S^p(\frac{1}{\sqrt{2}})\times S^{m-p}(\frac{1}{\sqrt{2}}))\to S^{m+1}\subset \r^{m+2}$ has two distinct principal curvatures $\lambda=-1$ with multiplicity $m_{\lambda}=p$ and $ \mu=1$ with  multiplicity $m_{\mu}=m-p=q$. It follows that the mean curvature of the hypersurface $H=(m-2p)/m$ is nonzero constant since $p\ne m/2$,  $|A|^2=m$, and $|A(\nabla f)|^2=|\nabla f|^2$. Substituting these into the  stability index form (\ref{Sm}) yields
\begin{equation}\label{Sm4}
\begin{aligned}
Q(f)=\int_{T^{p,q}} \big\{(\Delta f)^2 -4f\Delta f-4(q-p)^2 f^2\big\}{\rm d}v_g.
\end{aligned}
\end{equation}

For $1\le p< q=m-p\le m-1$, the Laplacian on $S^p(\frac{1}{\sqrt{2}})\times S^q(\frac{1}{\sqrt{2}})$ has spectrum:
\begin{equation}\notag
\cdots<\cdots<\cdots<\lambda_2=-2q< \lambda_1=-2p<\lambda_0=0,
\end{equation}
For an eigenfunction $f$ with eigenvalue $\nu=-2p$, i.e., $\Delta f=-2p f$, (\ref{Sm4}) reads
\begin{equation}\notag
\begin{aligned}
Q(f)&=\int_{T^{p,q}} [4p^2+8p-4(q-p)^2]f^2{\rm d}v_g\\
&=-4\int_{T^{p,q}} (q^2-2pq-2p)f^2{\rm d}v_g.
\end{aligned}
\end{equation}
It follows that $Q(f)<0$  if and only if the quadratic function $-4(q^2-2pq-2p)<0$, which is equivalent to  $q> p+\sqrt{p^2+2p}$. One can further check that $q> p+\sqrt{p^2+2p}$ is equivalent to $q> 2p$ since there is no integer within $(2p, p+\sqrt{p^2+2p})$.

On the other hand, for an eigenfunction $f$ of the second nonzero eigenvalue $\nu=-2q$, i.e., $\Delta f=-2q f$, (\ref{Sm4}) reads
\begin{equation}\notag
\begin{aligned}
Q(f)&=\int_{T^{p,q}} [4q^2+8q-4(q-p)^2]f^2{\rm d}v_g\\
&>\int_{T^{p,q}} 4(2q+p^2)f^2{\rm d}v_g> 0.
\end{aligned}
\end{equation}

Similarly, we can check that $Q(f)$ is positive on any other eigenspace of the Laplacian on $S^p(\frac{1}{\sqrt{2}})\times S^q(\frac{1}{\sqrt{2}})$.\\

Finally, since  $T^{p,q}=S^p(\frac{1}{\sqrt{2}})\times S^q(\frac{1}{\sqrt{2}})$ is compact, we have the following  Sturm-Liouville's decomposition
\begin{align}\notag
C^{\infty}(T^{p,q})=\oplus_{i=0}^{\infty} E_{\lambda_i},
\end{align}
where $E_{\lambda_i}$ denotes the eigenspace of the Laplacian on $S^p(\frac{1}{\sqrt{2}})\times S^q(\frac{1}{\sqrt{2}})$ with respect to the eigenvalue $\lambda_i$. \\

From this, together with the above discussion, we conclude that that for $1\le p<q=m-p\le2p$, the largest subspace of smooth functions on the biharmonic hypersurface $S^p(\frac{1}{\sqrt{2}})\times S^q(\frac{1}{\sqrt{2}})$ on which $Q(f)<0$ is $E_{\lambda_0}$, the eigenspaces of the Laplacian on $S^p(\frac{1}{\sqrt{2}})\times S^q(\frac{1}{\sqrt{2}})$ corresponding the eigenvalues $\lambda_0=0$. Since $E_{\lambda_0}=\r$ has dimension $1$. Thus, we obtain the first case in the index formula (\ref{indF}).  For the case $1\le p<q=m-p$ and $q>2p$, the largest subspace of smooth functions on the biharmonic hypersurface $S^p(\frac{1}{\sqrt{2}})\times S^q(\frac{1}{\sqrt{2}})$ on which $Q(f)<0$ is $E_{\lambda_0}\oplus E_{\lambda_1}$, where $E_{\lambda_0}$ and $E_{\lambda_1}$ are the eigenspaces of the Laplacian on $S^p(\frac{1}{\sqrt{2}})\times S^q(\frac{1}{\sqrt{2}})$ corresponding the eigenvalues $\lambda_0=0, \lambda_1=-2p$. Since the subspace $E_{\lambda_0}\oplus E_{\lambda_1}$ has dimension $1+(p+1)=p+2$,
we obtain the biharmonic  index of  $S^p(\frac{1}{\sqrt{2}})\times S^q(\frac{1}{\sqrt{2}})$ for the case  $1\le p<q=m-p$ and $q>2p$, which complete the proof of Statement (i).\\

 For Statement (ii), first note that the proper biharmonic hypersurface $S^m(\frac{1}{\sqrt{2}})\to S^{m+1}$  is totally umbilical with $|A|^2=m, H=-1$, and hence $A(\nabla f)=-\nabla f$. It follows that the index form  (\ref{Sm}) in this case reads
\begin{eqnarray}\notag
Q(f)&=\int_{S^m(\frac{1}{\sqrt{2}})} \big\{(\Delta f)^2 +4|\nabla f|^2-4m^2f^2\big\}{\rm d}v_g\\\label{Sm1}
&=\int_{S^m(\frac{1}{\sqrt{2}})} \big\{(\Delta f)^2 -4f \Delta f-4m^2f^2\big\}{\rm d}v_g,
\end{eqnarray}
where the second equality was obtained by using the divergence theorem.

Using the eigenvalues of the Laplacian on $S^m(\frac{1}{\sqrt{2}})$: 
\begin{align}\notag
\cdots<\cdots<\lambda_2<\lambda_1=-2m<\lambda_0=0.
\end{align}
one can check that for any eigenfunction function $f$ of the eigenvalue $\lambda_1=-2m$, we have $\Delta f=-2m f$, and hence (\ref{Sm1}) becomes
\begin{align}\notag
Q(f)&= \int_{S^m(\frac{1}{\sqrt{2}})} [(\Delta f)^2 -4f\Delta f-4m^2f^2]{\rm d}v_g=8m \int_{S^m(\frac{1}{\sqrt{2}})} f^2{\rm d}v_g>0.
\end{align}
Similarly, one can check that $Q(f)$ is positive on any other eigenspace, so the only subspace of $C^{\infty}(S^m(\frac{1}{\sqrt{2}}))$ on which $Q(f)<0$ is $\r$ which has dimension one. Thus, we have ${\rm Ind}\,(S^m(\frac{1}{\sqrt{2}}))=1$.
\end{proof}

\begin{remark}
Note that Statement (ii) in Theorem \ref{MT2} was  proved in \cite{LO05} in a quite different way.
\end{remark}
We end the paper with the following table which gives the indices of biharmonic hypersurfaces of  spheres in small dimensions.
\begin{center}
\begin{tabular}{|c|c|c|c|c|c|c|c|c|c|c|}
\hline Ambient sphere & Biharmonic hypersurface & Index\\
\hline $S^4$ & $S^1(\frac{1}{\sqrt{2}})\times S^2(\frac{1}{\sqrt{2}})$& 1   \\
\hline $S^5$ & $S^1(\frac{1}{\sqrt{2}})\times S^3(\frac{1}{\sqrt{2}})$& 3   \\
\hline $S^6$ & $S^1(\frac{1}{\sqrt{2}})\times S^4(\frac{1}{\sqrt{2}})$& 3   \\
\multirow{2}{4em} & $S^2(\frac{1}{\sqrt{2}})\times S^3(\frac{1}{\sqrt{2}})$& 1   \\ 
\hline $S^7$ & $S^1(\frac{1}{\sqrt{2}})\times S^5(\frac{1}{\sqrt{2}})$& 3   \\
\multirow{2}{4em} & $S^2(\frac{1}{\sqrt{2}})\times S^4(\frac{1}{\sqrt{2}})$& 1   \\ 
\hline $S^8$ & $S^1(\frac{1}{\sqrt{2}})\times S^6(\frac{1}{\sqrt{2}})$& 3   \\
\multirow{3}{4em} & $S^2(\frac{1}{\sqrt{2}})\times S^5(\frac{1}{\sqrt{2}})$& 4   \\ 
 & $S^3(\frac{1}{\sqrt{2}})\times S^4(\frac{1}{\sqrt{2}})$& 1   \\ 
\hline $S^9$ & $S^1(\frac{1}{\sqrt{2}})\times S^7(\frac{1}{\sqrt{2}})$& 3   \\
\multirow{3}{4em} & $S^2(\frac{1}{\sqrt{2}})\times S^6(\frac{1}{\sqrt{2}})$& 4   \\ 
 & $S^3(\frac{1}{\sqrt{2}})\times S^5(\frac{1}{\sqrt{2}})$& 1   \\ 
\hline $S^{10}$ & $S^1(\frac{1}{\sqrt{2}})\times S^8(\frac{1}{\sqrt{2}})$& 3   \\
\multirow{3}{4em} & $S^2(\frac{1}{\sqrt{2}})\times S^7(\frac{1}{\sqrt{2}})$& 4   \\ 
 & $S^3(\frac{1}{\sqrt{2}})\times S^6(\frac{1}{\sqrt{2}})$& 1   \\ 
 & $S^4(\frac{1}{\sqrt{2}})\times S^5(\frac{1}{\sqrt{2}})$& 1   \\ 
\hline
\end{tabular}
\end{center}
\vskip1cm

It follows from the table and Theorem \ref{MT2} that  for any natural number  $k$ except $k=2$, there exists a proper  biharmonic hypersurface in a sphere $S^m$ with $m$ depending on $k$ whose index is $k$.

\end{document}